\documentclass[11pt]{amsart}
\usepackage{amssymb}
\usepackage{lscape}
\usepackage{graphicx}

\newtheorem{theorem}{Theorem}[section]

\newtheorem{lemma}{Lemma}[section]
\newtheorem{cor}{Corollary}[section]

\newcommand{\Z}{{\mathbb Z}}

\newcommand{\C}{{\mathbb C}}
\newcommand{\R}{{\mathbb R}}

\begin{document}

\title[Polynomials with bounded height]{On the distribution of polynomials with bounded height}


\author[Csan\'ad Bert\'ok, Lajos Hajdu and Attila~Peth\H{o}]{Csan\'ad Bert\'ok, Lajos Hajdu$^*$ and Attila~Peth\H{o}$^{**}$}
\address{Institute of Mathematics\\ University of Debrecen \\
H-4002 Debrecen, P.O. Box 400, HUNGARY}
\email{hajdul@science.unideb.hu}
\email{bertok.csanad@science.unideb.hu}
\address{Department of Computer Science, University of Debrecen,
H-4002 Debrecen, P.O. Box 400, HUNGARY}
\email{Petho.Attila@inf.unideb.hu}
\thanks{$^*$Research supported in part by the NKFI grants K100339 and K115479.}
\thanks{$^{**}$Research supported in part by the NKFI grants K100339, NK104208 and K115479.}
\begin{abstract}
We provide an asymptotic expression for the probability that a randomly chosen polynomial with given degree, having integral coefficients bounded by some $B$, has a prescribed signature. We also give certain related formulas and numerical results along this line. Our theorems are closely related to earlier results of Akiyama and Peth\H{o}, and also yield extensions of recent results of Dubickas and Sha.
\end{abstract}

\date{\today}

\maketitle

\section{Introduction}


Let $d$ be a positive integer, $B\ge 1$ a real number. Denote by
$\mathcal H_d(B)$ the set of $(d+1)$-dimensional vectors $(p_0,\dots,p_d)$ satisfying $|p_i|\le B$ $(0\leq i\leq d)$. In the case $B=1$ we write simply $\mathcal H_d$ instead of $\mathcal H_d(1)$.

Given a polynomial $P \in \R[X]$,
the non-real roots of $P$ appear in complex conjugate pairs.
Thus $d=r+2s$, where $r$ denotes the number of real roots and $s$
the number of non-real pairs of roots of $P$.
As we shall work with arbitrary but fixed $d$ and then $r$ is uniquely determined by $s$, we call $s$ the signature of $P$.
The set $\mathcal H_d(B)$ splits naturally into $\lfloor d/2 \rfloor +1$ disjoint subsets
according to the signature. In the sequel $\mathcal H_d(s,B)$ denotes
the subset of $\mathcal H_d(B)$ whose elements have signature $s$. If $B=1$, in place of $\mathcal H_d(s,1)$ we shall simply write $\mathcal H_d(s)$. Plainly, $\mathcal H_d(s,B)$ is a bounded set in $\R^{d+1}$ for any $B>0$, and we will prove that it is Lebesgue measurable. For the Lebesgue measure (which we shall often simply call volume) of $A\subset \R^n$ we write $\lambda_n(A)$ or $\lambda(A)$, if the dimension $n$ is obvious.

Following Dubickas and Sha \cite{DuSh} denote by $\mathcal D_d^*(s,B)$\footnote{In fact Dubickas and Sha \cite{DuSh} called $(r,s)$ the signature of $P$ and used the notation $\mathcal D_d^*(s,B)$ instead of $\mathcal D_d^*(s,B)$. As we frequently cite the papers of Akiyama and Peth\H{o} \cite{AkiPet} and \cite{AkiPet2}, where only $s$ was used for the signature and sets of polynomials were denoted according to this convention, we follow their notation.} the set of polynomials $f(X)=p_dX^d +p_{d-1}X^{d-1}+\dots + p_0 \in \Z[X]$ with $p_d\not= 0$, $|p_i|\le B$ $(i=0,\dots,d)$ and such that $f$ has signature $s$. That is, $\mathcal D_d^*(r,s;B)=\mathcal H_d(s,B)\cap \mathbb Z[x]$. Denote by $D_d^*(s,B)$ the number of elements of $\mathcal D_d^*(s,B)$. They proved
\begin{equation}\label{DubSha}
B^{d+1} \ll D_d^*(s,B) \ll B^{d+1}
\end{equation}
by using a lower bound for the number of integer polynomials approximating appropriately a real polynomial of degree $d$ and signature $s$. They wrote: "It would be of interest to obtain an asymptotic formula as (1.1) in our setting as well." 

In this paper we improve considerably \eqref{DubSha} by providing an asymptotic formula for $D_d^*(s,B)$, thus we fulfill their request. It is important to mention that Akiyama and Peth\H{o} \cite{AkiPet, AkiPet2} considered a similar problem, when instead of the absolute values of the coefficients of the polynomials, the absolute values of the roots of the polynomials are assumed to be bounded. Our method works for other height functions too. For its application it is sufficient to prove that the boundary of the set of polynomials of height at most $B$ is a smooth function, see Lemma \ref{dave:thm}. Moreover one needs that the volume of the sets of polynomials with given signature of degree $d$ and of height $B$ is $\gg B^d$, see an example in the last section.

We also give a formula for $\lambda(\mathcal H_d(s,B))$ for any $d,s$ and $B$, involving integrals. Our formulas are similar to those obtained by Akiyama and Peth\H{o} \cite{AkiPet, AkiPet2}. Akiyama and Peth\H{o} could handle the integrals occurring there by Selberg integrals, and gave the precise volumes of the corresponding sets for small values of $d$. In our case, unfortunately we cannot handle the integrals theoretically, except certain 'small' cases. To get some numerical results we apply the Monte Carlo method to approximate the occurring integrals for $d\leq 15$.

The structure of the paper is the following. In the next section we give our theoretical results. Then we prove our theorems. In the fourth section our numerical results are given for $d\leq 15$. Finally, we indicate some open problems.

\section{New results}

Our main result is the following.

\begin{theorem} \label{main}
  We have
  $$
  D_d(s,B) = \lambda_{d+1}(\mathcal H_d(s)) B^{d+1} + O(B^{d}).
  $$
Moreover, $\lambda_{d+1}(\mathcal H_d(s))>0$ for all $d$ and $s$.
\end{theorem}

In our proof we follow closely the ideas of Akiyama and Peth\H{o} \cite{AkiPet2}. Our main tool is a classical result of Davenport \cite{D1}, which quantifies the ancient principle that if we blow up by a factor $B$ a $d$-dimensional set with appropriate properties then the number of lattice points is approximately $B^d$ times the volume of the original set.

Our further aim is to derive a formula for the volume of $\mathcal H_d(s,B)$. For this purpose we need some preparation. Denote by $S_j(x_1,\ldots,x_d)$ $(j=1,\ldots,d)$ the $j$-th elementary symmetric polynomial of $x_1,\ldots,x_d$, that is
$$S_j(x_1,\ldots,x_d)=\sum_{1\le i_1<\cdots <i_j\le d} x_{i_1}\cdots x_{i_j}.$$
For later use we define $S_0(x_1,\ldots,x_d)=1$. For $B>0$ let $H_d(s,B)$ denote the set of such $d$-dimensional real vectors $(x_1,\ldots,x_r,y_1,\ldots,y_s,z_1,\ldots,z_s)$ which satisfy the inequalities
$$-B\le S_j(x_1,\ldots,x_r,y_1+iz_1,y_1-iz_1,\ldots,y_s+iz_s,y_s-iz_s)\le B\ (1\le j\le d)$$
and $z_j\not= 0$ $(j=1,\dots,s)$, where $i=\sqrt{-1}$. If $B=1$, we simply write $H_d(s)$ for this set.

Obviously, we have $(p_0,\dots,p_d)\in \mathcal H_d(s,B)$ if and only if the vector $(x_1,\ldots,x_r,y_1,\ldots,y_s,z_1,\ldots,z_s)$ belongs to $H_d(s,B)$, where $x_1,\dots,x_r,y_1\pm z_1i, y_s\pm z_si$ are the roots of $p_dX^d +\dots+p_0$.

Denote by Res$(P(X),Q(X))$ the resultant of $P(X),Q(X)\in\mathbb{R}[X]$. For any possible $s$ and positive real number $B$ put
$$
\mathcal H_d^*(s,B):=\{(p_0,\dots,p_{d-1})\in{\mathbb R}^d\ :\ (p_0,\dots,p_{d-1},1)\in \mathcal H_d(s,B)\},
$$
and
$$
\mathcal H_d^+(s,B):=\{(p_0,\dots,p_d)\in \mathcal H_d(s,B)\ :\ p_d>0\}.
$$
When $B=1$, the above sets are simply denoted by $\mathcal H_d^*(s)$ and $\mathcal H_d^+(s)$, respectively.

By the above notion, we have the following theorem.
\begin{theorem}\label{thm:2}
Let $R_j(X)=X^2-2y_jX+y^2_j+z^2_j$ $(j=1,\ldots,s)$. Then
	$$\lambda_d(\mathcal H_d^*(s,B))=\frac{2^s}{r!s!}\int\limits_{H_d(s,B)}|\Delta_r|\Delta_s\Delta_{r,s}\prod_{j=1}^{s}|z_j|\ dx_1\ldots dx_rdy_1dz_1\ldots dy_sdz_s,$$
	where
	\begin{align*}
	\Delta_r&=\prod_{1\le j<k\le r} (x_j-x_k),\\
	\Delta_s&=\prod_{1\le j<k\le s} {\rm{Res}}(R_j(X),R_k(X)),\\
	\Delta_{r,s}&=\prod_{j=1}^{r}\prod_{k=1}^{s}R_k(x_j).
	\end{align*}
	
	Furthermore, we have
	$$\lambda_{d+1}(\mathcal H_d^+(s,B))=B^{d+1}\int\limits_{0}^{1}u^d\lambda\left(\mathcal H_d^*\left(s,\frac{1}{u}\right)\right)du.$$
\end{theorem}

We note that by Theorem \ref{main} we know that $\lambda_{d+1}(\mathcal H_d(s,B))$ exists for any $B>0$. Further, in view of $\lambda_{d+1}(\mathcal H_d(s,B))=2\lambda_{d+1}(\mathcal H_d^+(s,B))$ (see Corollary \ref{halmazc} below), the above theorem gives a formula (though implicit) for $\lambda_{d+1}(\mathcal H_d(s,B))$, for any $B>0$.

\section{Proofs}

In this section we prove our theorems. First we investigate
$\mathcal{H}_d^+(s,B)$ for $B>0$. Later we also need to consider the set $\mathcal{H}_d^-(s,B)$, which is the set of vectors $\bf v$, such that $-{\bf v} \in \mathcal{H}_d^+(s,B)$.

\begin{lemma} \label{halmaz}
   The set $\mathcal{H}_d^+(s)$ has positive Riemann measure and its boundary is the union of finitely many algebraic surfaces.
\end{lemma}

\begin{proof}
Following Akiyama and Peth\H{o} \cite{AkiPet2}, denote $\mathcal E_d^{(s)}(B)$ $(s=0,\dots,\lfloor d/2 \rfloor)$ the set of vectors $(p_{0},\dots,p_{d-1})\in \R^d$ such that the corresponding polynomial $X^d+p_{d-1}X^{d-1}+\dots+p_0$ has signature $s$, and all of its roots lie in the disc of radius $B$.

Let $P(X) = p_dX^d + p_{d-1}X^{d-1}+\dots+p_0$ with $0<p_d\le 1$, $|p_j|\le 1$, $(j=0,\dots,d-1)$ and with signature $s$. The mapping $Y=p_d X$ is continuous, and we have $p^{d-1}P(X) = Q(Y) = Y^d + p_{d-1}Y^{d-1}+p_{d-2}p_dY^{d-2}+\dots+p_0p_d^{d-1}$, moreover the signatures of $P(X)$ and $Q(Y)$ are equal. By Proposition 2.5.9. of \cite{MiSt} all roots of $Q(Y)$ lie in the disc of radius $2$. Thus $(p_0p_d^{d-1},\dots,p_{d-2}p_d,p_{d-1}) \in \mathcal F_d^{(s)} (p_d)$, where
$$\mathcal F_d^{(s)} (p_d) := \mathcal E_d^{(s)} (2)\cap ([-p_d^{d-1},p_d^{d-1}]\times  \dots \times [-p_d,p_d]\times [-1,1]).$$
By Lemma 2.1 of \cite{AkiPet2}, $\mathcal E_d^{(s)}(B)$ is Riemann measurable for any $B>0$, thus $\mathcal F_d^{(s)} (p_d)$ is Riemann measurable, as well. Denote by $F_d^{(s)} (p_d)$ its $d$-dimensional Riemann measure. The function $F_d^{(s)} (p_d)$ is continuous for $p_d>0$, because $\mathcal E_d^{(s)} (2)$ is independent of $p_d$, and its boundary is by Theorem 7.1. of \cite{AkiPet} the union of finitely many algebraic surfaces. Also, the box $[-p_d^{d-1},p_d^{d-1}]\times  \dots \times [-p_d,p_d]\times [-1,1]$ depends continuously on $p_d$. Thus we have
$$
\lambda_{d+1}(\mathcal H_d^+(s)) = \lim_{p_d \to 0}\int\limits_{p_d}^1 p_d^{d(d-1)/2}F_d^{(s)}(p_d) dp_d.
$$
As $F_d^{(s)} (p_d)$ is continuous for $p_d>0$, this integral exists.

Now we prove that $\lambda_{d+1}(\mathcal H_d^+(s))$ is positive. For this purpose, assume $1/2\le p_d\le 1$ in the rest of this proof. (The argument works with any positive lower bound for $p_d$, but to prove our claim the choice $1/2$ is sufficient.) Assume that $q_0,\dots, q_{d-1}$ are so small that all roots of $Q(Y)=Y^d +q_{d-1}Y^{d-1}+\dots+q_0$ lie in the disc with radius $4^{-d}$. Then it is an easy exercise to show, that $|q_j| \le 2^{-d+j+1}\le p_d^{-d+j+1}$ $(j=0,\dots,d-1)$. Thus the inverse image of $\mathcal E_d^{(s)}(4^{-d})$ lie in $\mathcal H_d^+(s)$. Thus $F_d^{(s)} (p_d)\ge \lambda_d(\mathcal E_d^{(s)}(4^{-d}))>0$ for all $p_d\ge 1/2$. Thus
$$
\lambda_{d+1}(\mathcal H_d^+(s)) \ge \frac12 \left(\frac12\right)^{d(d-1)/2} \lambda_d(\mathcal E_d^{(s)}(4^{-d})),
$$
which is certainly a positive number.

Let $p_d X^d +p_{d-1}X^{d-1}+\dots+p_0$ be a polynomial with indeterminate coefficients lying in a commutative ring. Then its discriminant $D= D(p_0,\dots,p_d)$ is a homogenous polynomial in $p_0,\dots,p_d$ of degree $d(d-1)$. Specializing the coefficient domain to $\C$ it is well-known that $D=0$ if and only if either $p_d=0$, or $p_d\not= 0$ and the polynomial has multiple roots. Using the later fact Akiyama and Peth\H{o} \cite{AkiPet} proved, see Theorem 7.1., that the inner boundary points of $\mathcal E_d^{(s)}(1)$ lie on the hypersurface $S_D$ defined by the equation $D=0$. Repeating that proof to $\mathcal{H}_d^+(s)$, one can see that its boundary is the union of finitely many pieces of $S_D$ and the intersection of the hyperplane $p_d=0$ with the hypercube $[-1,1]^{d+1}$.
\end{proof}

\begin{cor} \label{halmazc}
  Let $B>0$. Then $\mathcal{H}_d^+(s,B)$ and $\mathcal{H}_d^-(s,B)$ have positive Riemann measure and their boundaries are the union of finitely many algebraic surfaces. Moreover
  $$
  \lambda(\mathcal{H}_d^+(s,B)) = \lambda(\mathcal{H}_d^-(s,B)) =\lambda(\mathcal{H}_d^+(s)) B^{d+1}.
  $$
\end{cor}

\begin{proof}
  The assertion follows directly from Lemma \ref{halmaz} together with the fact that $(x_0,\dots,x_d) \in \mathcal{H}_d^+(s)$ if and only if $(Bx_0,\dots,Bx_d) \in \mathcal{H}_d^+(s,B)$.
\end{proof}

The basic ingredient of the proof of Theorem \ref{main} is the following result of Davenport.

\begin{lemma}[{\cite[Theorem]{D1}}]
\label{dave:thm}
Let $\mathcal{R}$ be a closed bounded region in the $n$ dimensional space $\R^n$
and let $\mathrm{N}(\mathcal{R})$ and $\lambda(\mathcal{R})$ denote the number of points
with integral coordinates in $\mathcal{R}$ and the volume of $\mathcal{R}$,
respectively. Suppose that:
\begin{itemize}
\item Any line parallel to one of the $n$ coordinate axes intersects
  $\mathcal{R}$ in a set of points which, if not empty, consists of at most $h$
  intervals.
\item The same is true (with $m$ in place of $n$) for any of the $m$
  dimensional regions obtained by projecting $\mathcal{R}$ on one of the
  coordinate spaces defined by equating a selection of $n-m$ of the
  coordinates to zero; and this condition is satisfied for all $m$
  from $1$ to $n-1$.
\end{itemize}
Then
\[
\lvert\mathrm{N}(\mathcal{R})-\lambda(\mathcal{R})\rvert\leq\sum_{m=0}^{n-1}h^{n-m}V_m,
\]
where $V_m$ is the sum of the $m$ dimensional volumes of the
projections of $\mathcal{R}$ on the various coordinate spaces obtained by
equating any $n-m$ coordinates to zero, and $V_0=1$ by convention.
\end{lemma}

Now we can give the proof of our main result.

\begin{proof}[Proof of Theorem \ref{main}]
For any $B>0$ we have
$$
\mathcal{H}_d(s,B) = \mathcal{H}_d^+(s,B) \cup \mathcal{H}_d^-(s,B) \cup (\{0\}\times [-B,B]^d).
$$
Thus, by Lemma \ref{halmazc} we get
$$
 \lambda(\mathcal{H}_d(s,B)) = 2 \lambda(\mathcal{H}_d^+(s,B)) = 2 \lambda(\mathcal{H}_d^+(s)) B^{d+1}.
$$
We use Lemma \ref{dave:thm} with the choice $\mathcal R = \mathcal{H}_d(s,B)$ and $n=d+1$. Clearly, $\mathcal{D}_d^*(s,B) = \mathcal{H}_d(s,B) \cap \Z^{d+1}$, hence
$$
\mathrm{N}(\mathcal{R}) = D_d^*(s,B).
$$
Thus to apply Lemma \ref{dave:thm} we have to ensure that its assumptions hold for $\mathcal{H}_d(s,B)$. First of all, $\mathcal{H}_d(s,B)$ is a bounded set because it lies in the box $[-B,B]^{d+1}$. By Lemma \ref{halmazc} our set $\mathcal{H}_d(s,B)$ is Lebesgue measurable and its boundary is the union of finitely many algebraic surfaces. Thus by the remark after the proof of the Theorem of \cite{D1} the assumptions of Lemma \ref{dave:thm} are satisfied. Thus
$$
|D_d^*(s,B) - 2 \lambda(\mathcal{H}_d^+(s,B))| \le \sum_{m=0}^{d} h^{d+1-m} V_m,
$$
where $h$ is the maximal number of intervals obtained when we intersect $\mathcal{H}_d(s,B)$ with a line parallel to one of the coordinate axes.

Such a line $\ell$ admits a  parametrization of the form $(x_0+ty_0,\dots,x_d+ty_d)$, where $(x_0,\dots,x_d)$, $(y_0,\dots,y_d)\in \R^{d+1}$ are fixed and $t$ runs through $\R$. Inserting this parametrization into the equation $D=0$, the left hand side becomes a non-zero polynomial of degree at most $d(d+1)$ in $t$. Thus $\ell$ intersects the hypersurface $D=0$ in at most $d(d+1)$ points, which partition $\ell$ into at most $d(d+1)+1$ intervals, i.e. $h\le d(d+1)+1$.

As $V_m$ is the $m$-dimensional volume of a subset of $[-B,B]^m$, we have $V_m =O(B^m)$. This implies
$$
|D_d^*(s,B) - 2 \lambda(\mathcal{H}_d^+(s,B))| \le \sum_{m=0}^{d} (d(d+1)+1)^{d+1-m} O(B^m) = O(B^{d}),
$$
and our theorem is proved.
\end{proof}

Now we give the proof of our second theorem.

\begin{proof}[Proof of Theorem \ref{thm:2}]
	The first statement concerning the formula given for $\lambda(\mathcal{H}^*_d(s,B))$ follows by a simple calculation from Theorem 2.1. of \cite{AkiPet}.
	
	To prove the formula for $\lambda(\mathcal H_d^+(s,B))$ we start from
	$$\lambda(\mathcal H_d^+(s,B))=\int\limits_{\mathcal H_d^+(s,B)}1\ dp_0\ldots dp_d.$$
	We apply the substitution
	$$p_d=Bq_d,\hspace{3mm}p_i=Bq_dq_i\hspace{3mm}(i=0,\ldots,d-1).$$
	Observe that the determinant of its Jacobian is $B^{d+1}q_d^d$. Thus we have
	$$\lambda(\mathcal H_d^+(s,B))=B^{d+1}\int\limits_{A} q_d^d\ dq_0\ldots dq_d,$$
	where
	$$A=\{(q_0, \dots, q_{d-1},q_{d})\in{\mathbb R}^{d+1}\ :\ X^d+q_{d-1}X^{d-1}+\dots+q_1X+q_0$$
	$$
	\text{has signature}\ s\ \text{and}\
	0<q_d\leq 1,\ -\frac{1}{q_d}\leq q_i\leq \frac{1}{q_d}\ (i=0,\dots,d-1)\}.$$
	Here we used the trivial fact that the signatures of the polynomials
	$$X^d+q_{d-1}X^{d-1}+\dots+q_1X+q_0$$
	and
	$$Bq_dX^d+Bq_dq_{d-1}X^{d-1}+\dots+Bq_dq_1X+Bq_dq_0$$
	are the same.
	Putting everything together, we have
	$$\lambda_{d+1}(\mathcal H_d^+(s,B))=B^{d+1} \int\limits_0^1 q_d^d \lambda_d\left(\mathcal H_d^*\left(s,\frac{1}{q_d}\right)\right)\ dq_d$$
	which proves the theorem.
\end{proof}

\section{Numerical results}

In this section we give some numerical data regarding $\lambda(\mathcal{H}^*_d(s))$ and $\lambda(\mathcal{H}^+_d(s))$. We can calculate the precise values of $\lambda(\mathcal{H}^*_d(s))$ only for $d=2,3$, and of $\lambda(\mathcal{H}^+_d(s))$ only for $d=2$. Evaluating the integrals appearing in Theorem \ref{thm:2} we obtain
$$
\lambda(\mathcal{H}^*_2(0))=\frac{13}{6}=2.1667,\ \lambda(\mathcal{H}^*_2(1))=\frac{11}{6}=1.8333,
$$
$$
\lambda(\mathcal{H}^*_3(0))=\frac{766}{1215}+\frac{\log (3)}{6}=0.8136,\ \lambda(\mathcal{H}^*_3(1))=\frac{8954}{1215}-\frac{\log (3)}{6} =7.1865,
$$
and
$$
\lambda(\mathcal{H}^+_2(1))=\frac{31}{18}-\frac{1}{3} \log (2)=1.4912,\
\lambda(\mathcal{H}^+_2(0))=\frac{41}{18}+\frac{1}{3} \log (2) = 2.5088.
$$

Here and later on, to perform our calculations we used the program package Mathematica, and the values are always given with four digit precision.

Observe that $\lambda(\mathcal{H}^*_2(s))$ $(s=0,1)$ are rational, but $\lambda(\mathcal{H}^+_2(s))$ and $\lambda(\mathcal{H}^*_3(s))$ $(s=0,1)$ are transcendental. We think that $\lambda(\mathcal{H}^+_d(s))$ and $\lambda(\mathcal{H}^*_d(s))$ $(s=0,\dots,\lfloor d/2\rfloor)$ are all transcendental for $d\ge 2$ and $d\ge 3$, respectively. In contrast, Akiyama and Peth\H{o} \cite{AkiPet}, Theorem 5.1., proved that the analogous values $v_d^{(s)}$ are rational for all $d,s$.

For larger values of $d$ we were unable to evaluate the integrals appearing in Theorem \ref{thm:2}. The reason is that when we split up the original domain into subdomains according to the signature, the boundary (coming from the discriminant surface) is so complicated that Mathematica is not able to handle the situation. So to get some numerical data, we needed another approach. We used the Monte Carlo method to get approximate results both for $\lambda(\mathcal{H}^\ast_d(s))$ and $\lambda(\mathcal{H}^+_d(s))$ for $2\leq d\leq 15$.
The main principle behind the method is that we choose a 'large' number of randomly generated polynomials inside the basic region, and check their signatures. Then the ratio of polynomials having a prescribed signature $s$ gives an approximation of the volume. More precisely, we do the following.
\begin{enumerate}
	\item For approximating $\lambda(\mathcal{H}^*_d(s))$, we randomly choose (using uniform distribution) a vector from $[-1,1]^d$, say $(p_0,\ldots,p_{d-1})$. For approximating $\lambda(\mathcal{H}^+_d(s,1))$ we do the same, but now the vector is in $[0,1]\times [-1,1]^{d}$.
	\item We construct the polynomial $P(X)=X^d+p_{d-1}X^{d-1}+\cdots +p_1X+p_0$ or $P(X)=p_dX^d+p_{d-1}X^{d-1}+\cdots +p_1X+p_0$, respectively.
	\item We determine the signature of $P(X)$.
	\item After a 'large' number of iterations (in our case we used $200,000$ loops for each $d$) we calculate the ratio of the number of polynomials with a given signature and the number of iteration, which is the approximate value of $\lambda(\mathcal{H}^*_d(s))$ or $\lambda(\mathcal{H}^+_d(s))$, respectively.
\end{enumerate}

In the following tables we give the results of the above method for $2\leq d\leq 15$, that is, the approximate values of $\lambda(\mathcal{H}^*_d(s))$ and $\lambda(\mathcal{H}^+_d(s))$, respectively (for all possible values of $s$). We note that comparing the approximate values with the precise values given above for $d=2,3$ and $d=2$, respectively, we see that in those cases the errors are around $1\%$. Thus we expect that the other approximate values are rather close to the actual data, as well.

\begin{table}[!ht]
	\caption{The approximated values of $\lambda(\mathcal{H}^*_d(s))$ for $2\leq d\le 15$}
	\makebox[\linewidth][c]{$\begin{array}{|c|c|c|c|c|c|c|c|c|}
	\hline
	d/s&0&1&2&3&4&5&6&7\\\hline
	2&2.1652&1.8348&-&-&-&-&-&-\\\hline
	3&0.8192&7.1808&-&-&-&-&-&-\\\hline
	4&0.0880&10.2833&5.6286&-&-&-&-&-\\\hline
	5&0.0021&6.3378&25.6602&-&-&-&-&-\\\hline
	6&0.0003&1.6330&43.9437&18.4230&-&-&-&-\\\hline
	7&0.0000&0.1542&34.128&93.7178&-&-&-&-\\\hline
	8&0.0000&0.0051&12.4442&179.8340&63.7171&-&-&-\\\hline
	9&0.0000&0.0000&2.0838&163.8780&346.0380&-&-&-\\\hline
	10&0.0000&0.0000&0.1434&72.8678&728.5040&222.4840&-&-\\\hline
	11&0.0000&0.0000&0.0102&16.0154&744.4378&1287.5366&-&-\\\hline
	12&0.0000&0.0000&0.0000&1.6589&382.8122&2909.0406&802.4883&-\\\hline
	13&0.0000&0.0000&0.0000&0.0410&98.0173&3227.6070&4866.3347&-\\\hline
	14&0.0000&0.0000&0.0000&0.0000&10.6496&1847.4598&11599.2986&2926.5920\\\hline
	15&0.0000&0.0000&0.0000&0.0000&0.8192&574.4230&13800.5709&18392.1869\\\hline
	\end{array}$}
\end{table}
\newpage
\begin{table}[!ht]
	\caption{The approximated values of $\lambda(\mathcal{H}^+_d(s))$ for $2\le d\le 15$}
	\makebox[\linewidth][c]{$\begin{array}{|c|c|c|c|c|c|c|c|c|}
		\hline
		d/s&0&1&2&3&4&5&6&7\\\hline
		2&2.5054&1.4946&-&-&-&-&-&-\\\hline
		3&1.7540&6.2460&-&-&-&-&-&-\\\hline
		4&0.6301&11.3332&4.0361&-&-&-&-&-\\\hline
		5&0.1061&10.7558&21.1381&-&-&-&-&-\\\hline
		6&0.0128&5.5776&45.8112&12.5984&-&-&-&-\\\hline
		7&0.0013&1.6326&52.2074&74.1587&-&-&-&-\\\hline
		8&0.0000&0.2163&33.6922&180.6090&41.4822&-&-&-\\\hline
		9&0.0000&0.0154&12.4595&232.6550&266.8700&-&-&-\\\hline
		10&0.0000&0.0051&2.6317&171.8940&706.3810&143.0890&-&-\\\hline
		11&0.0000&0.0000&0.3174&74.3629&998.0621&975.2576&-&-\\\hline
		12&0.0000&0.0000&0.0000&18.2886&814.0595&2754.4576&509.1942&-\\\hline
		13&0.0000&0.0000&0.0000&2.6214&400.1792&4165.5501&3263.5493&-\\\hline
		14&0.0000&0.0000&0.0000&0.4096&123.2896&3719.1680&10721.5258&1819.6070\\\hline
		15&0.0000&0.0000&0.0000&0.0000&22.1184&1965.7523&17215.8157&13564.3136\\\hline
		\end{array}$}
\end{table}

The graphs of the functions $\lambda(\mathcal{H}^*_d(s))$, $\lambda(\mathcal{H}^+_d(s))$ and of $v_d^{(s)}$ from \cite{AkiPet} seem to have similar fashion. For small $s$ their values tend rapidly to zero. This was proved for $v_d^{(0)}$ in \cite{AkiPet} Theorem 6.1. and for $v_d^{(1)}$ by Kirschenhofer and Weitzer \cite{KW}.

In the following figures we illustrate our results in a more comprehensive way. On the top of the bars we indicate the values of $s$.
\begin{figure}[!ht]
	\caption{Approximate values of $\lambda(\mathcal{H}^*_d(s))$ for even $d<15$}
	\makebox[\linewidth][c]{
	\includegraphics[scale=0.75]{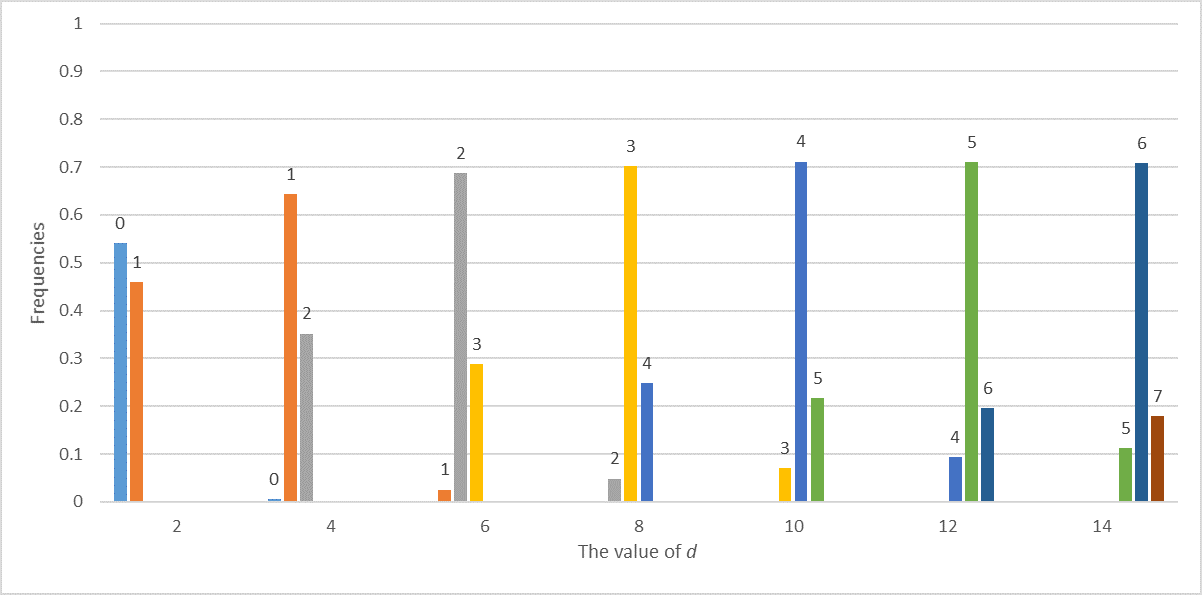}}
\end{figure}
\vfill
\begin{figure}[!ht]
	\caption{Approximate values of $\lambda(\mathcal{H}^*_d(s))$ for odd $d\leq 15$}
	\makebox[\linewidth][c]{
		\includegraphics[scale=0.75]{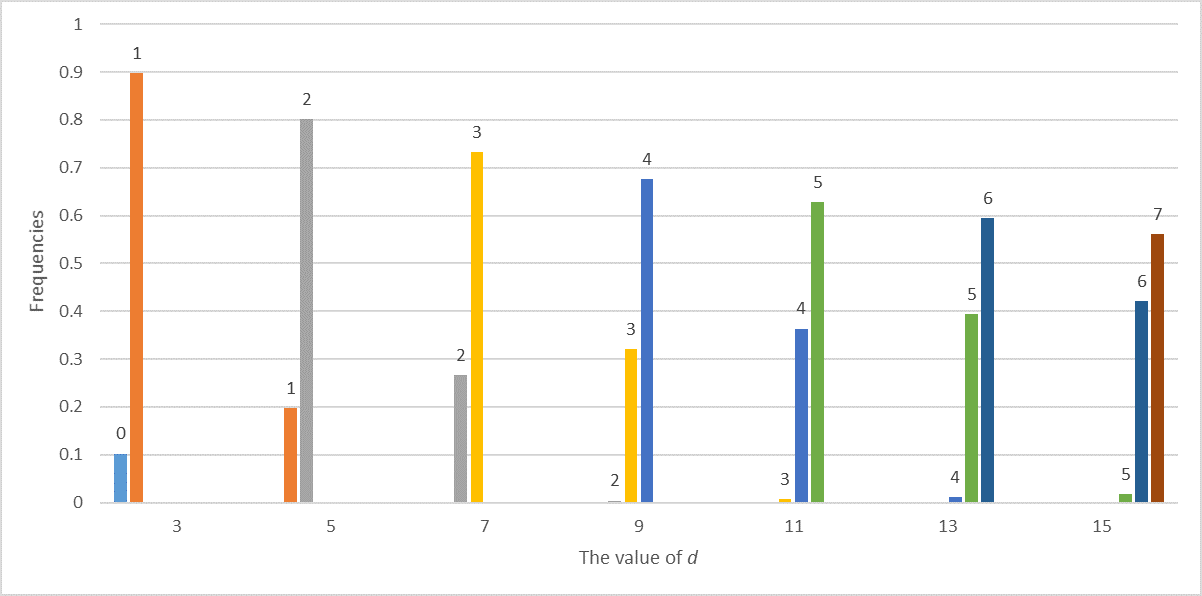}}
\end{figure}
\begin{figure}[!ht]
	\caption{Approximate values of $\lambda(\mathcal{H}^+_d(s))$ for even $d<15$}
	\makebox[\linewidth][c]{
		\includegraphics[scale=0.75]{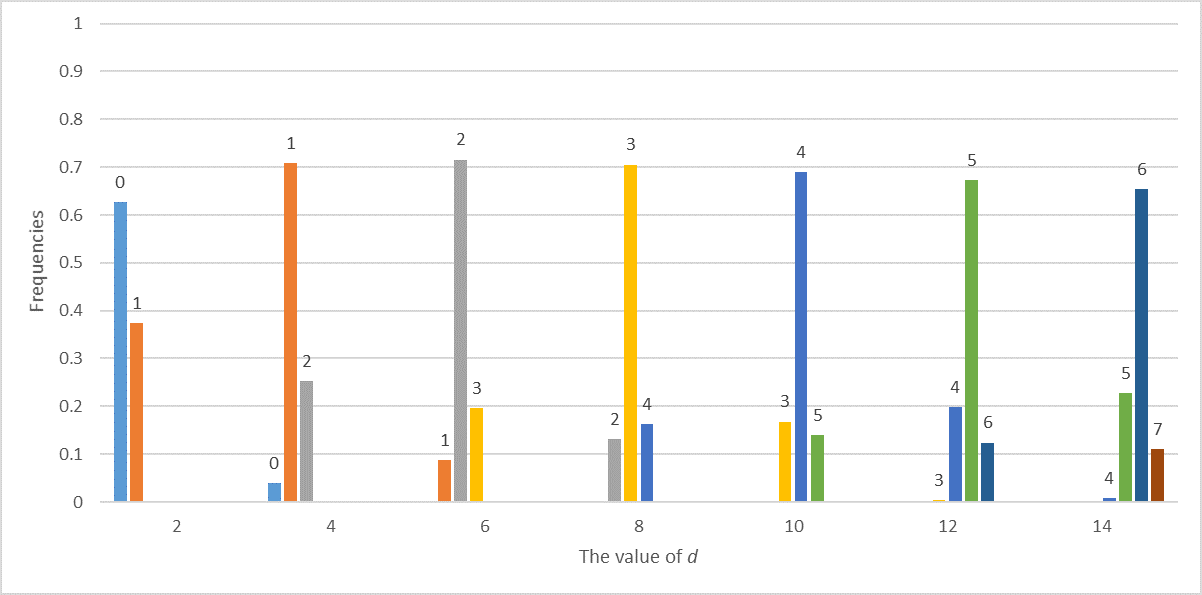}}
\end{figure}
\clearpage
\begin{figure}[pt]
	\caption{Approximate values of $\lambda(\mathcal{H}^+_d(s))$ for odd $d\leq 15$}
	\makebox[\linewidth][c]{
		\includegraphics[scale=0.75]{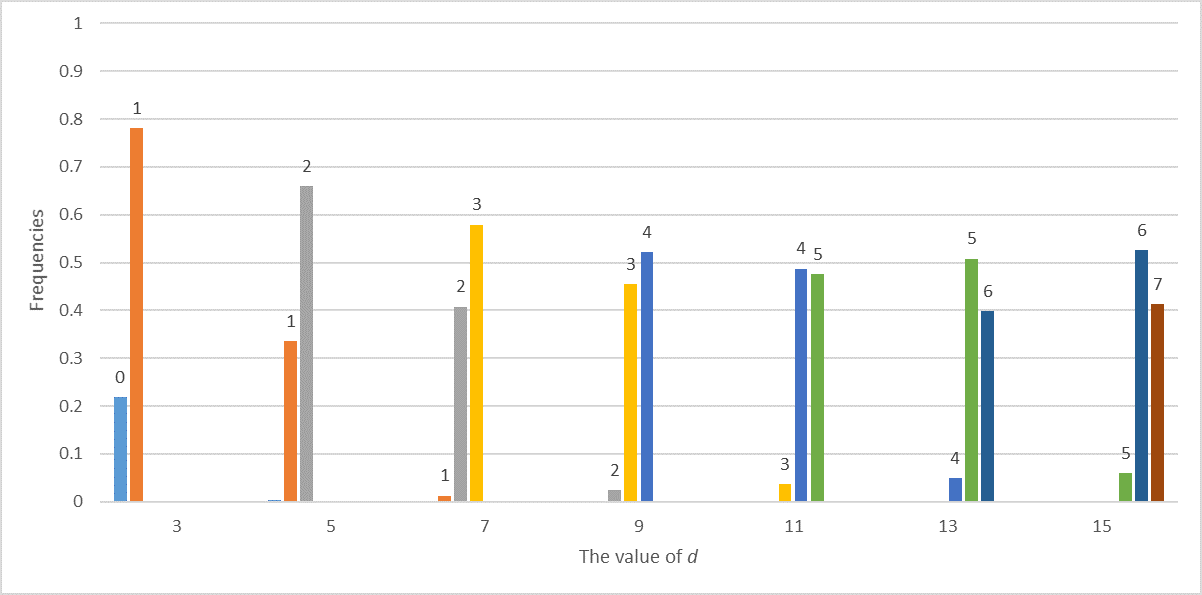}}
\end{figure}
\vfill
\section{Open problems}

In Theorem \ref{main} we proved an asymptotic formula for $D_d^*(s,B)$. It is natural to ask whether a similar formula holds for the number of integer polynomial with bounded height, but with leading coefficient $1$. More formally, denote by $\mathcal D_d(s,B)$ the subset of elements of $\mathcal D_d^*(s,B)$ with $p_d=1$, and denote by $D_d(s,B)$ the size of $\mathcal D_d(s,B)$. Dubickas and Sha \cite{DuSh} proved
\begin{equation}\label{DubSha1}
B^{d+1} \ll D_d^*(s,B) \ll B^{d+1}.
\end{equation}
We expect analogously that
$$
B^{d} \ll D_d(s,B) \ll B^{d}
$$
hold, but we were not able to prove this estimate. Of course, the upper bound follows from the trivial identity $\sum_{s=0}^{[d/2]}  D_d(s,B) = (2\lfloor B\rfloor+1)^d$, the real challenge is to establish the lower bound. One way to achieve this is to prove that $\lambda_d(\mathcal H_d^*(s,B))/B^d >c$ with a positive constant $c$, which is independent of $B$. Unfortunately we were not able to prove this.

Similarly, we could not prove the simpler statement $\lambda_d(H_d(s,B))/B^d >c$ with a positive constant $c$, which is independent of $B$. The main problem seems to be to find efficient construction of integer polynomials with given signature, with leading coefficient $1$, and with large height.

\bigskip

The discriminant hypersurface $S_D$ defined by $D=0$ partitions $\R^d$ into subsets such that the polynomials arising by the mapping $(p_0,\dots,p_{d-1}) \mapsto X^d + p_{d-1}X^{d-1} + \dots+ p_0 $ have the same signature if $(p_0,\dots,p_{d-1})$ runs through the points of a subset. It is natural to ask the topology of these subsets, e.g. whether they are connected. The situation is very simple for $d=2$, when $D = p_1^2 - 4p_0$. Thus $S_D$ is a parabola, which partitions the plain into two subsets.


\begin{thebibliography}{99}
\bibitem{AkiPet}
S. Akiyama and A. Peth\H{o}, {\em On the distribution of polynomials with bounded roots, I. Polynomials with real coefficients}, J. Math. Soc. Japan, {\bf 66} (2014), 927—949.

\bibitem{AkiPet2}
S. Akiyama and A. Peth\H{o}, {\em On the distribution of polynomials with bounded roots, II. Polynomials with integer coefficients}, Unif. Distrib. Theory, {\bf 9} (2014), 5 – 19.

\bibitem{D1}
H. Davenport, {\em On a principle of Lipschitz.} J. London Math. Soc. {\bf 26},
(1951). 179--183. {\em Corrigendum} ibid {\bf 39} (1964), 580.

\bibitem{DuSh}
A. Dubickas and M. Sha, {\em Positive density of integer polynomials with some prescribed properties}, J. Number Theory, 159 (2016), 27 -- 44.

\bibitem{monte}
J. Hammersley and D. C. Handscomb, {\em Monte Carlo Methods}, Springer Netherlands (1964), 178 pp, DOI: 10.1007/978-94-009-5819-7

\bibitem{KW}
P. Kirschenhofer, and M. Weitzer, {\em A number theoretic problem on the distribution of polynomials with bounded roots}, Integers 15 (2015), Paper No. A10, 10 pp.

\bibitem{MiSt}
M. Mignotte and D. \c{S}tef\v{a}nescu, {\em Polynomials An Algebraic Approach}, Springer-Verlag Singapure, 1999.

\end{thebibliography}
\end{document}